\documentclass[11pt]{article}

\RequirePackage[OT1]{fontenc}
\RequirePackage[amsthm,amsmath,natbib]{imsart}
\RequirePackage[dvips,colorlinks]{hyperref}
\usepackage{amssymb}
\usepackage{graphicx}
\usepackage[latin1]{inputenc}
\usepackage{multirow}

 \newcommand{\eps}{\varepsilon}

 \newcommand{\p}{\mathbb{P}}
 \newcommand{\E}{\mathbb{E}}

 \newcommand{\F}{\mathcal{F}}

 \newcommand{\Real}{\mathbb{R}}

\newtheorem{theorem}{Theorem}[section]
\newtheorem{corollary}{Corollary}[section]
\newtheorem{proposition}{Proposition}[section]

\theoremstyle{definition}

\newtheorem{remark}{Remark}[section]

\newcounter{hypoconbis}
\newcounter{saveconbis}
\newcommand\debutA{\begin{list} {\textbf{A\arabic{hypoconbis}}}{\usecounter{hypoconbis}}\setcounter{hypoconbis}{\value{saveconbis}}}
\newcommand\finA{\end{list}\setcounter{saveconbis}{\value{hypoconbis}}}

\newcounter{hypocom}
\newcounter{savecom}
\newcommand{\debutB}{\begin{list}{\textbf{B\arabic{hypocom}}}{\usecounter{hypocom}}\setcounter{hypocom}{\value{savecom}}}
\newcommand{\finB}{\end{list}\setcounter{savecom}{\value{hypocom}}}

\newcounter{saveeqn}

\newlength{\defaultheadheight}
\newlength{\defaultheadsep}
\begin{document}

\begin{frontmatter}
\title{A strong law of large numbers for martingale arrays}
\runtitle{SLLN for martingales arrays}

\vspace{1cm}

{\small {\em Yves F. Atchad\'e\footnote{Department of Statistics, University of Michigan, email:
yvesa@umich.edu}}\\[5pt]
(March 2009)} \\[10pt]

\begin{abstract}
We prove a martingale triangular array generalization of the Chow-Birnbaum-Marshall's inequality. The result is used to derive a strong law of large numbers for martingale triangular arrays whose rows are asymptotically stable in a certain sense. To illustrate, we derive a simple proof, based on martingale arguments, of the consistency of kernel regression with dependent data. Another application can be found in \cite{atchadeetfort08} where the new inequality is used to prove a strong law of large numbers for adaptive Markov Chain Monte Carlo methods.
\end{abstract}

\begin{keyword}[class=AMS]
\kwd[Primary ]{60J27, 60J35, 65C40}
\end{keyword}

\begin{keyword}
\kwd{Martingales and Martingale arrays}
\kwd{Strong law of large numbers}
\kwd{Kernel regression}
\end{keyword}

\end{frontmatter}

\section{Strong law of large numbers for martingale arrays}\label{intro}
Let $(\Omega,\F,\p)$ be a probability space and $\E$ the expectation operator with respect to $\p$. Let $\{D_{n,i},\F_{n,i},\;1\leq i\leq n\}$, $n\geq 1$ be a martingale-difference array. That is for each $n\geq 1$, $\{\F_{n,i},\;1\leq i\leq n\}$ is a non-decreasing sequence of sub-sigma-algebra of $\F$, for any $1\leq i\leq n$, $\E\left(|D_{n,i}|\right)<\infty$ and $\E\left(D_{n,i}|\F_{n,i-1}\right)=0$. We assume throughout the paper that $\F_{n,0}=\{\emptyset,\Omega\}$ for all $n\geq 0$. We introduce the partial sums
\[M_{n,k}:= \sum_{i=1}^k D_{n,i},\;\;1\leq k\leq n,\;n\geq 1.\]
For each $n\geq 1$, $\{(M_{n,k},\F_{n,k}),\;1\leq k\leq n\}$ is a martingale. Let $\{c_n,\;n\geq 1\}$ be a non-increasing sequence of positive numbers. We are interested in conditions under which $c_nM_{n,n}$ converges almost surely to zero.

Martingales and martingale arrays play an important role in Probability and Statistics as valuable tools for limit theory. Much is known on the limit theory of martingales (see e.g. \cite{halletheyde80}) but comparatively little work has been done on the law of large numbers for martingale arrays. 

One of the most effective approach to proving the strong law of large numbers for martingales is via the Kolmogorov's inequality for martingales obtained by Chow (\cite{chow60}) and Birnbaum-Marshall (\cite{birnbaumetmarshall60}).
\begin{theorem}[Chow-Birnbaum-Marshall's inequality]
Let $\{\left(S_k,\F_k\right),\; k \geq 1\}$ be a sub-martingale and $\{c_k, k\geq
1\}$ a non-increasing real-valued sequence. For $p \geq 1$ and $n \leq N$
\[
\p\left( \sup_{n \leq m \leq N} c_m |S_m| \geq 1 \right) \leq c_N^p \E[|S_N|^p] + \sum_{m=n}^{N-1}
(c_m^p -c_{m+1}^p) \; \E[|S_m|^p].
\]
\end{theorem}

The following theorem gives an extension to martingale arrays. We introduce the sequence
\[S_{n,k}=M_{n,k}=\sum_{i=1}^k D_{n,i},\;\;\; 1\leq k\leq n \;\;\; \mbox{ and }\;\;\; S_{n,k}=\sum_{i=1}^n D_{n,i}+\sum_{j=n+1}^kD_{j,j},\;\;\;k>n.\]
\[R_n:= \sum_{j=1}^{n-1}\left(D_{n,j}-D_{n-1,j}\right).\]
\begin{theorem}\label{thm2}
Let $\{D_{n,i},\F_{n,i},\;1\leq i\leq n\}$, $n\geq 1$ be a martingale-difference array and $\{c_n,\;n\geq 1\}$ a non-increasing sequence of positive numbers. Assume that $\F_{n,i}=\F_i$ for all $i,n$. For $n\leq m\leq N$, $p\geq 1$ and $\lambda>0$
\begin{multline}2^{-p}\lambda^p\p\left(\max_{n\leq m\leq N}c_m|M_{m,m}|>\lambda\right)\leq c_N^p\E\left(|S_{n,N}|^p\right)+\sum_{j=n}^{N-1}\left(c_j^p-c_{j+1}^p\right)\E\left(|S_{n,j}|^p\right) \\
+ \E\left[\left(\sum_{j=n+1}^N c_j|R_j|\right)^p\right].\end{multline}
\end{theorem}
\begin{proof}For $n\leq m\leq N$, we have $M_{m,m}=M_{m-1,m-1}+D_{m,m}+R_m$ leading to the decomposition
\[M_{m,m}=M_{n,n}+\sum_{j=n+1}^mD_{j,j}+\sum_{j=n+1}^mR_j=S_{n,m}+\sum_{j=n+1}^mR_j.\]
We note that $\{(S_{n,m},\F_m),\;1\leq m\leq N\}$ is a martingale. We also introduce
\[Z_m=c_n^p|S_{n,n}|^p+\sum_{j=n+1}^mc_j^p\left(|S_{n,j}|^p-|S_{n,j-1}|^p\right).\]
It is easy to check that $Z_m$ has the alternative form
\[Z_m=c_m^p|S_{n,m}|^p+\sum_{j=n}^{m-1}\left(c_j^p-c_{j+1}^p\right)|S_{n,j}|^p.\]
Since $\{(|S_{n,m}|^p,\F_m),\;1\leq m\leq N\}$ is a sub-martingale and $\{c_k,\;k\geq 1\}$ is non-increasing, we have $\E\left(Z_m|\F_{m-1}\right)\geq Z_{m-1}$, that is $\{(Z_m,\F_m),\;n\leq m\leq N\}$ is a sub-martingale. For $n\leq m\leq N$, we introduce the sets $A_m^{(1)}:=\{c_m|S_{n,m}|>\lambda/2\}$, $A_m^{(2)}:=\{c_m|\sum_{j=n+1}^m R_j|>\lambda/2\}$ and $B_m:=\{c_j|M_{j,j}|\leq \lambda, j=n,\ldots,m-1,\; c_m|M_{m,m}|>\lambda\}$. We have:
\begin{eqnarray*}
\lambda^p\p\left(\max_{n\leq m\leq N}c_m|M_{m,m}|>\lambda\right)&=&\E\left(\sum_{m=n}^N\lambda^p\textbf{1}_{B_m}\right)\leq \E\left(\sum_{m=n}^N\lambda^p\textbf{1}_{B_m\cap A_m^{(1)}}\right)+\E\left(\sum_{m=n}^N\lambda^p\textbf{1}_{B_m\cap A_m^{(2)}}\right)\\
&\leq & \E\left(\sum_{m=n}^N2^p|c_mM_{m,m}|^p\textbf{1}_{B_m\cap A_m^{(1)}}\right)+\E\left(\sum_{m=n}^N2^p|c_m\sum_{j=n+1}^m R_j|^p\textbf{1}_{B_m\cap A_m^{(2)}}\right)\\
&\leq & \E\left(\sum_{m=n}^N2^pZ_m\textbf{1}_{B_m\cap A_m^{(1)}}\right)+2^p\E\left[\left(\sum_{j=n+1}^N c_j|R_j|\right)^p\right]\\
&\leq & \E\left[\sum_{m=n}^N2^p\E\left(Z_N\textbf{1}_{B_m\cap A_m^{(1)}}|\F_m\right)\right]+2^p\E\left[\left(\sum_{j=n+1}^N c_j|R_j|\right)^p\right]\\
&\leq & 2^p\E\left[Z_N + \left(\sum_{j=n+1}^N c_j|R_j|\right)^p\right].
\end{eqnarray*}
\end{proof}

In many situations, one deals with martingale arrays whose rows are asymptotically stable in the sense that the sequence $\E\left[|R_n|\right]$ converges to zero as $n$ increases to infinity. Theorem \ref{thm2} can be used to prove a strong law of large numbers for such martingale arrays.
\begin{corollary}\label{cor1}
Let $\{D_{n,i},\F_{n,i},\;1\leq i\leq n\}$, $n\geq 1$ be a martingale-difference array and $\{c_n,\;n\geq 1\}$ a non-increasing sequence of positive numbers. Assume that $\F_{n,i}=\F_i$ for all $i,n$. Suppose that there exists $p\geq 1$ such that for any $n_0\geq 1$
\begin{equation}\label{assump}
\lim_{n\to\infty}\left(c_n^p\E\left[|S_{n_0,n}|^p\right]+ \sum_{k=n}^\infty\left(c_k^p-c_{k+1}^p\right)\E\left[|S_{n,k}|^p\right]\right)=0,\;\; \mbox{ and }\;\; \sum_{n\geq 1} c_n\E^{1/p}\left[|R_n|^p\right]<\infty.\end{equation}
Then $c_nM_{n,n}$ converges almost surely to zero.\end{corollary}

\begin{remark}
With respect to the process $\{R_n\}$ in Theorem \ref{thm2}, we point out that, because of the assumption $\F_{n,j}=\F_j$, the sequence $\{\sum_{j=1}^{k}\left(D_{n,j}-D_{n-1,j} \right),\F_k,\; 1\leq k\leq n-1\}$ is a also martingale.
\end{remark}

The conditions in Corollary \ref{cor1} are expressed in terms of moments of martingales. These moments can be nicely bounded by moments of the martingale differences. We give one such bound in the next proposition. It is a consequence of the Burkholder's inequality (\cite{halletheyde80}, Theorem 2.10) and some classical convexity inequalities. We omit the details.
\begin{proposition}\label{prop1}Let $\{D_{n,i},\F_{n,i},\;1\leq i\leq n\}$, $n\geq 1$ be a martingale-difference array. For any $p>1$,
\begin{equation}\label{boundM}
\E\left[\left|M_{n,k}\right|^p\right]\leq Ck^{\max(p/2,1)-1}\sum_{j=1}^k\E\left(\left|D_{n,j}\right|^p\right),\end{equation}
where $C=\left(18pq^{1/2}\right)^p$, $p^{-1}+q^{-1}=1$.
\end{proposition}

\section{Kernel regression with Markov chains}
As an application, we prove the strong consistency of the Nadaraya-Watson estimator for nonparametric regression where the data arises from a non-stationary Markov chain. The approach of the proof can be adapted to study other kernel methods or other statistical smoothing procedures with dependent data. We assume the following structure for the data. $\{(X_i,\epsilon_i),\;i\geq 0\}$ is a joint $\Real^2$-valued Markov chain on some probability space $(\Omega,\F,\p)$ such that
\[\p\left((X_n,\epsilon_n)\in A\times B|(X_k,\epsilon_k),\;k\leq n-1\right)=\int_A p(X_{n-1},z)q(z,B)dz,\]
for transition probability densities $p$ and $q$. $p$ is the transition probability density of the marginal Markov chain $\{X_i,\;i\geq 0\}$ and $q(x,A)=p\left(\eps_n\in A| X_n=x\right)$ is the transition probability density of the error term $\eps_n$. All densities are with respect to the Lebesgue measure denoted $dx$. We assume that $p$ has an invariant distribution $\pi$ (that is $\pi(x)=\int_{\Real}\pi(y)p(y,x)dy$, $x\in\Real$) and
\begin{equation}\label{meanerror}\int_\Real \pi(x)\left(\int_\Real \epsilon q(x,\epsilon)d\epsilon\right)dx=0.\end{equation}
We consider the dependent variable
\[Y_i=r(X_i)+\epsilon_i,\;\;\;i\geq 0.\]
We are interested in estimating the regression function $r$. Note that the error terms $\epsilon_i$ are correlated and we do not assume that $\E\left(\epsilon_i\right)= 0$ unless, as assumed in  (\ref{meanerror}), the Markov chain $\{X_i,\;i\geq 0\}$ is in stationarity. For the reader's convenience, we will sometimes use the notation $\E(U(Y)|X=x)$ to denote the integral $\int_\Real U\left(r(x)+\epsilon\right)q(x,\epsilon)d\epsilon$, whenever such integral is well-defined.  A popular nonparametric estimator for $r$ is the Nadaraya-Watson estimator
\begin{equation}\hat r_n(x_0)=\frac{\sum_{i=1}^nY_iK\left(\frac{x_0-X_i}{h_n}\right)}{\sum_{i=1}^nK\left(\frac{x_0-X_i}{h_n}\right)},\;\;\;x_0\in\Real.\end{equation}
where $K$ is the kernel (a nonnegative function such that $\int_\Real K(x)dx=1$) and $h_n>0$ the bandwidth. Let $\psi:\;\Real\to\Real$ be a measurable function. We study the almost sure convergence of
\begin{equation*}\label{estimator}\hat r_{n,\psi}(x_0)=\frac{1}{nh_n}\sum_{i=1}^n\psi(Y_i)K\left(\frac{x_0-X_i}{h_n}\right),\end{equation*}
as $n\to\infty$. We can then deduce the convergence of the Nadaraya-Watson estimator by setting $\psi(x)=x$ for the numerator and $\psi(x)=1$ for the denominator.

Let $\mu$ be the distribution of $X_0$, the initial distribution of the Markov chain. We write $P$ for the Markov kernel induced by $p$ which operates on nonnegative bounded measurable functions as $Pf(x)=\int_{\Real}p(x,y)f(y)dy$. The iterates operators of $P$ are defined as $P^0f(x)=f(x)$ and for $n\geq 1$, $P^nf(x)=P(P^{n-1}f)(x)$. We will assume that $P$ is geometrically ergodic. That is

\debutB
\item \label{B1} $P$ is $\phi$-irreducible, aperiodic and there exist a function $V:\; \Real\to [1,\infty)$, $\lambda\in (0,1)$, $b\in (0,\infty)$ such that
\[PV(x)\leq \lambda V(x)+b\textbf{1}_\mathcal{C}(x),\]
for some small set $\mathcal{C}$.
\finB

This assumption is a well known stability assumption for Markov kernels extensively studied in \cite{meynettweedie93}. One important consequence of  (B\ref{B1}) that we will use is the following. For any $\alpha\in (0,1]$, there exists $C(\alpha)<\infty$ such that for all $n\geq 0$,
\begin{equation}\label{geoergo}\sup_{|f|_{V^\alpha}\leq 1}\left|P^nf(x)-\int_{\Real}f(x)\pi(x)dx\right|\leq C(\alpha)\rho^nV^\alpha(x),\;\;\;\;x\in\Real\end{equation}
where $|f|_{V^\alpha}:=\sup_{x\in\Real}\frac{|f(x)|}{V^\alpha(x)}$. A proof can be found \cite{meynettweedie93}, Chapter 15.

We assume that $\mu(V):=\int_{\Real}V(x)\mu(x)dx<\infty$. By iterating the drift condition (B\ref{B1}), it is easy to see that
\begin{equation}\label{boundV}
\sup_{n\geq 0}\E\left[V(X_n)\right]\leq \mu(V)+b/(1-\lambda)<\infty.\end{equation}

On the function $\psi$, we assume that
\begin{equation}\label{CondFun}
\sup_{x\in\Real}V^{-1/2}(x)\left(1+|x|\right)\E\left[\left|\psi(Y)\right||X=x\right]<\infty,\;\;\;\mbox{ and }\;\;\;\sup_{x\in\Real}V^{-1}(x)\E\left[\psi^2(Y)|X=x\right]<\infty.\end{equation}
On the kernel $K$, we assume that
\begin{equation}\label{CondKernel}
\sup_{x\in\Real}K(x)<\infty,\;\;\;\; \lim_{|x|\to\infty}|x|K(x)=0,\;\;\;\mbox{ and }\;\;\;\sup_{x\neq x'}\frac{\left|K(x)-K(x')\right|}{|x-x'|}<\infty.\end{equation}
On the sequence $\{h_n,\; n\geq 0\}$, we assume that:
\begin{equation}\label{CondSeq}
h_n\sim n^{-\beta},\;\;\; \mbox{ with }\;\;\;\;\beta\in (0,1/4).\end{equation}

\begin{theorem}\label{thm3}Assume (B\ref{B1}), (\ref{CondFun}-\ref{CondSeq}) and that the function $x\to\pi(x)\E\left(\psi(Y)|X=x\right)$ is continuous at $x_0$. Then $\lim_{n\to\infty}\hat r_{n,\psi}(x_0)=\pi(x_0)\E\left(\psi(Y)|X=x_0\right)$ with $\p$-probability one.
\end{theorem}
\begin{proof}Throughout the proof, $x_0\in\Real$ is fixed and $C$ will denote a finite constant whose actual value might differ from one appearance to the next. Define $\F_n:=\sigma((X_k,Y_k),\;k\leq n\}$. For $h>0$, define $F_h(x,y)=\psi(y)K\left(\frac{x_0-x}{h}\right)$, $f_h(x)=K\left(\frac{x_0-x}{h}\right)\E\left(\psi(Y)|X=x\right)$, and
\[g_h(x)=\sum_{l\geq 0}\bar P^lf_h(x),\]
where $\bar P^lf(x):=P^lf(x)-\int_\Real f(x)\pi(x)dx$. By (\ref{CondFun}), the boundedness of $K$ and the geometric ergodicity assumption (\ref{geoergo}), $g_h$ is well-defined and satisfies $|g_h|_{V^{1/2}}\leq C(1-\rho)^{-1}$. It is also well-known that $g_h$ solves the Poisson equation for $f_h$ and $P$. In other words, we have
\begin{equation}\label{poisson1}
g_h(x)-P g_h(x)=\bar f_h(x),\end{equation}
 where $\bar f_h(x)=f_h(x)-\int_\Real f_h(x)\pi(x)dx$.

Similarly, define $H_h(x,y)=F_h(x,y)+P g_h(x)$. It is left to the reader to check that \\$\E\left[H_h(X_{n},Y_{n})\vert X_{n-1}=x,Y_{n-1}=y\right]=Pf_h(x)+P^2g_h(x)=Pg_h(x)+\int_\Real f_h(x)\pi(x)dx$ (using (\ref{poisson1})). It follows that
\begin{equation}\label{poisson2}F_h(x,y)-\int \pi(x)f_h(x)dx=H_h(x,y)-\E\left[H_h(X_{n},Y_{n})\vert X_{n-1}=x,Y_{n-1}=y\right],\;x,y\in\Real.\end{equation}
Using (\ref{poisson2}), we can decompose $\hat r_{n,\psi}(x_0)$ as
\begin{multline*}
\hat r_{n,\psi}(x_0)=\frac{1}{h_n}\int_\Real K\left(\frac{x_0-x}{h_n}\right)\E\left[\psi(Y)|X=x\right]\pi(x)dx+\frac{1}{n h_n}\sum_{k=1}^nD_{n,k}\\
+(nh_n)^{-1}\left(\E\left[H_{h_n}(X_1,Y_1)|\F_0\right]-\E\left[H_{h_n}(X_{n+1},Y_{n+1})|\F_n\right]\right),\end{multline*}
where $D_{n,k}=H_{h_n}(X_k,Y_k)-\E\left[H_{h_n}(X_{k},Y_{k})\vert \F_{k-1}\right]$.

Under the stated assumptions, it is a standard result of kernel estimation that \[\lim_{n\to\infty}\frac{1}{h_n}\int_\Real K\left(\frac{x_0-x}{h_n}\right)\E\left[\psi(Y)|X=x\right]\pi(x)dx=\E\left(\psi(Y)|X=x_0\right)\pi(x_0).\]
See e.g. \cite{prakasarao83} for a proof.

We deduce from the drift condition (B\ref{B1}) and (\ref{CondFun}) that
\[\sup_{n\geq 1}\E\left[|H_h(X_n,Y_n)||\F_{n-1}\right]\leq C V^{1/2}(X_{n-1}),\]
for some finite constant $C$ that does not depend on $h$. Combined with (\ref{boundV}) we get for any $\delta>0$,
\[\sum_{k\geq 1}\p\left((nh_n)^{-1}\left|\E\left[H_{h_n}(X_1,Y_1)|\F_0\right]-\E\left[H_{h_n}(X_{n+1},Y_{n+1})|\F_n\right]\right|>\delta\right)\leq C\delta^{-2}\sum_{n\geq 1}n^{-2(1-\beta)}<\infty.\]
This easily implies that the term $(nh_n)^{-1}\left(\E\left[H_{h_n}(X_1,Y_1)|\F_0\right]-\E\left[H_{h_n}(X_{n+1},Y_{n+1})|\F_n\right]\right)$ converges almost surely to zero.

Lastly, the process $\{(D_{n,k},\F_k)\;k\leq n\}$ is a martingale-difference array. Again by (\ref{CondFun}), the boundedness of $K$, the drift condition (B\ref{B1}), $\E\left(|D_{n,k}|^2\right)\leq \E\left[\left|H_{h_n}(X_k,Y_k)\right|^2\right]\leq C\E\left(V(X_k)\right)$. Then using (\ref{boundV}), we obtain that $\sup_{n\geq 0}\sup_{0\leq k\leq n}\E\left[\left|D_{n,k}\right|^2\right]<\infty$. This implies, in the notations of Theorem \ref{thm2}, that $\E\left[|S_{n,m}|^2\right]\leq C m$, for some finite constant $C$ that does not depend on $n$ nor $m$. Moreover, we can write  $D_{n,j}-D_{n-1,j}=H_{h_n}(X_j,Y_j)-H_{h_{n-1}}(X_j,Y_j)-\E\left(H_{h_n}(X_j,Y_j)-H_{h_{n-1}}(X_j,Y_j)\vert \F_{j-1}\right)$, and we note that
\begin{multline*}
\left(H_{h_n}-H_{h_{n-1}}\right)(x,y)=\psi(y)\left(K\left(\frac{x_0-x}{h_n}\right)-K\left(\frac{x_0-x}{h_{n-1}}\right)\right)\\
+\sum_{l\geq 1}\bar P^l\left(K\left(\frac{x_0-x}{h_n}\right)-K\left(\frac{x_0-x}{h_{n-1}}\right)\right)\E\left[\psi(Y)\vert X=x\right].\end{multline*}
By the Lipschitz condition on $K$ and (\ref{CondFun}),
\[\left|\E\left[\psi(Y)\vert X=x\right]\left(K\left(\frac{x_0-x}{h_n}\right)-K\left(\frac{x_0-x}{h_{n-1}}\right)\right)\right|\leq C\left|h_{n-1}^{-1}-h_n^{-1}\right|V^{1/2}(x).\]
Therefore $\left|\left(H_{h_n}-H_{h_{n-1}}\right)(x,y)\right|\leq C\left|h_{n-1}^{-1}-h_n^{-1}\right|\left(|x|\psi(y)+V^{1/2}(x)\right)$ from which we deduce using (\ref{CondFun}) and (\ref{boundV}) that $\E\left(\left|D_{n,j}-D_{n-1,j}\right|^2\right)=O(n^{2(-1+\beta)})$ uniformly in $j$ which implies  as in Proposition \ref{prop1} that $\E^{1/2}\left(\left|\sum_{j=1}^{n-1}D_{n,j}-D_{n-1,j}\right|^2\right)=O(n^{-1/2+\beta})$ which together with $\E\left[|S_{n,m}|^2\right]\leq C m$ proves (\ref{assump}), since $\beta<1/4$. We can therefore conclude that $(n h_n)^{-1}\sum_{k=1}^nD_{n,k}\to 0$, $\p$-almost surely, which ends the proof.
\end{proof}

\bibliographystyle{ims}
\bibliography{biblio}

\begin{thebibliography}{6}
\expandafter\ifx\csname natexlab\endcsname\relax\def\natexlab#1{#1}\fi
\expandafter\ifx\csname url\endcsname\relax
  \def\url#1{\texttt{#1}}\fi
\expandafter\ifx\csname urlprefix\endcsname\relax\def\urlprefix{URL }\fi

\bibitem[{Atchade and Fort(To appear)}]{atchadeetfort08}
\textsc{Atchade, Y.~F.} and \textsc{Fort, G.} (To appear).
\newblock Limit theorems for some adaptive mcmc a;gorithms with sub-geometric
  kernels.
\newblock \textit{Bernoulli} .

\bibitem[{Birnbaum and W.(1961)}]{birnbaumetmarshall60}
\textsc{Birnbaum, Z.~W.} and \textsc{W., M.~A.} (1961).
\newblock Some multivarite chebyshev inequalities with extensions to continuous
  parameter processes.
\newblock \textit{Ann. Math. Statist.} \textbf{32} 687--703.

\bibitem[{Chow(1960)}]{chow60}
\textsc{Chow, Y.~S.} (1960).
\newblock A martingale inequality and the law of large numbers.
\newblock \textit{Proc. Amer. Math. Soc.} \textbf{11} 107--111.

\bibitem[{Hall and Heyde(1980)}]{halletheyde80}
\textsc{Hall, P.} and \textsc{Heyde, C.~C.} (1980).
\newblock \textit{Martingale Limit theory and its application}.
\newblock Academic Press, New York.

\bibitem[{Meyn and Tweedie(1993)}]{meynettweedie93}
\textsc{Meyn, S.~P.} and \textsc{Tweedie, R.~L.} (1993).
\newblock \textit{Markov chains and stochastic stability}.
\newblock Springer-Verlag London Ltd., London.

\bibitem[{Prakasa(1983)}]{prakasarao83}
\textsc{Prakasa, B. L.~S., R.} (1983).
\newblock \textit{Nonparametric functional estimation}.
\newblock Academic Press, New York.

\end{thebibliography}

\end{document}